\documentclass{article}
\usepackage{graphicx}
\usepackage{amsmath}
\usepackage{amsfonts}
\usepackage{amsthm}
\usepackage{hyperref}
\usepackage{listings}
\usepackage{clrscode3e}
\newtheorem{theorem}{Theorem}[section]
\newtheorem{theorem*}{Theorem*}
{\tiny {\normalsize }}

\newcommand{\seqnum}[1]{\href{https://oeis.org/#1}{\underline{#1}}} 

\begin{document}

\title{Integer Sequences: Irregular Arrays 
\\ and Intra-Block Permutations}

\author{Boris Putievskiy}
\date{October 27, 2023}
\maketitle
\begin{abstract}  This article investigates integer sequences that partition the sequence into blocks of various lengths - irregular arrays. The main result of the article is explicit  formulas for numbering of irregular arrays. A generalization of Cantor diagonal method is proposed. We also define and describe intra-block permutations of natural numbers. Generalizations of reluctant sequences are introduced, namely generalized reluctant sequences and generalized reverse reluctant sequences. Explicit formulas are presented for these sequences. The article provides numerous examples to illustrate all statements.
\end{abstract}

\tableofcontents
\section{Introduction}
Denote the set of integers by $\mathbb{Z}$, the set of nonnegative integers by $\mathbb{Z}^{*}$, the set of positive integers by $\mathbb{Z}^{+}$, the set of positive real numbers by $\mathbb{R}^{+}$.  Denote the set of integer sequences by $\mathcal{A}$ and the set of positive integer sequences by $\mathcal{A}^{+}$.
\\A pairing  function is a function that reversibly maps $\mathbb{Z}^{+}$ x $\mathbb{Z}^{+}$ $\rightarrow$ $\mathbb{Z}^{+}$.
A permutation of natural numbers  is bijective map $\mathbb{Z}^{+}$ $\rightarrow$ $\mathbb{Z}^{+}$. 
\\ A block (or segment) of a sequence is any set of the consecutive terms of the form $(a_{k+1},a_{k+2},a_{k+3},\: ... \: a_{k+m})$, where $k\in \mathbb{Z}^{*},\:\: m \in \mathbb{Z}^{+}$, and $m$  is the length of the block.
\\Throughout this paper, we will refer to sequences by their $Annnnn$ numbers, as found in the Online Encyclopedia of Integer Sequences [1]. Denote the sequence of natural numbers $(1,2,3,...)\:\: \seqnum{A000027} \:\: \text{by}\:\: \xi $.
%
%
\section{Partitions of the Set of Positive Integers}
{\bf Definition 2.1.}
Let a sequences $\alpha$: $a_{1}, a_{2}, a_{3},...\in \mathcal{A}$ \text{and} $\beta$: $b_{1}, b_{2}, b_{3},...$ $\in$ $\mathcal{A}^{+}$.
\\The sequence $\beta$ partitions the sequence $\alpha$ of into blocks of lengths $b_{1}, b_{2}, b_{3},...$. The sequence $\alpha$ is written as irregular array read by rows:
$\\ a_{1}, a_{2}, ... a_{b_{1}},$ 
$\\ a_{b_{1}+1}, a_{b_{1}+2}, ... a_{b_{1}+b_{2}},$
$\\ a_{b_{1}+ b_{2}+1},a_{b_{1}+ b_{2}+2}, ... a_{b_{1}+b_{2}+b_{3}},$
$\\ \quad .\quad .\quad.\\$
The sequence $\beta$ is called partitioning sequence. We use two parameters to number the terms of an irregular array $L(n)$ and $R(n)$. Where $L(n)$ represents the block number, and $R(n)$  indicates the position within the block from left to right. Thus
$\\ (1,1),(1,2), ... (1,b_{1}),$ 
$\\ (2,1),(2,2), ... (2,b_{2}),$ 
$\\ (3,1),(3,2), ... (3,b_{3}),$
$\\ \quad .\quad .\quad.\\$
Denote by $B(s)=b_{1}+b_{2}+...+b_{s}$ partial sums $\beta$, 
\begin{equation}
B(s)=0, \quad B(s-1)+b_{s}=B(s).
\end{equation} 
Let $L(0)=0, \:\: \text{for} \:\:  n \geq 1 \:\: \text{we get}$: 
\begin{equation}
R(n)=n - B(L(n)-1).
\end{equation} 
Denote by $R^{'}(n)$ the position within the block from right to left. Then 
$$\\R^{'}(n) = B(L(n))+1-n, \:\: R(n)+R^{'}(n)=b_{L(n)}+1.$$ 
Using (2), we can derive a formula for the inverse problem: how to calculate the number of terms if the values of the functions $L$ and $R$ are known. 
$$\\n=B(L-1)+R.$$
Let the sequences $\beta = \xi$, then a sequence $\alpha$ is written as regular array read by rows:
$\\ a_{1}, $ 
$\\ a_{1}, a_{2},$
$\\ a_{1},a_{2},a_{3},$
$\\ \quad .\quad .\quad.\\$
These formulas are commonly known  $\seqnum{A003056}, \:\: \seqnum{A002260}, \:\:  \seqnum{A004736}.$
Row numbering of a regular array starts from 0: $t=\lfloor \dfrac{\sqrt{8n-7}-1}{2} \rfloor.$
\\Then $ L(n)=t+1,$
$$\\R(n)= n-\dfrac {t(t+1)}{2}, \:\: R^{'}(n)= \dfrac {(t+1)(t+2)}{2}+1-n,\:\: R(n)+R^{'}(n)=t+2.$$
\begin{theorem}  
Let $ x(n): \mathbb{Z}^{+} \rightarrow \mathbb{R}^{+}\: \text{and} \:\: x(n) \:\: \text{is the largest  root of the equation} 
\\B(x)=n.$ Then 
\begin{equation}
	L(n)=\lceil x(n) \rceil.
\end{equation}
\end{theorem}
\begin{proof}
	By definition $B(0) = 0, \: B(1) =b_{1}, \: B(2) =b_{1}+b_{2}, \: ...$ The function $B(n) \:$ is strictly increasing. Therefore
	\begin{align*}
		0 &< x(1) < x(2) < ... < x(b_1) = 1, \\
		1 &< x(b_1 + 1) < x(b_1 + 2) < ... < x(b_1 + b_2) = 2, \\
		2 &< x(b_1 + b_2 + 1) < x(b_1 + b_2 + 2) < ... < x(b_1 + b_2 + b_3) = 3,
	\end{align*}
	$$\cdots$$
	We obtain
			\begin{align*}
				 & \lceil x(1) \rceil=1,	\lceil x(2) \rceil=1,	 ... ,	\lceil x(b_{1}) \rceil=1, \\
				 & \lceil x(b_{1}+1) \rceil=2,		\lceil x(b_{1}+2) \rceil=2,		... ,		\lceil x(b_{1}+b_{2}) \rceil=2,\\
				 & \lceil x(b_{1}+b_{2}+1) \rceil=3, \lceil x(b_{1}+b_{2}+2) \rceil=3, ... ,  \lceil x(b_{1}+b_{2}+b_{3}) \rceil=3,
			\end{align*}
			$$\cdots$$
\end{proof}

Let $L(n)$  be the number of block of the sequence $\:\:\beta: 
b_{1}, b_{2}, b_{3},... \in \mathcal{A}^{+} 
\\ \text{and}\:\: m \in \mathbb{Z}^{+},\:\: m>1.$ The following properties hold.
\\\textbf{(P2.1.)} The number of the block of the sequence $:\: mb_{1}, mb_{2}, mb_{3},...$ is
$L(u)$,
\\ $\text{where}\:\: u=\Bigl\lfloor \frac{n-1}{m} \Bigr\rfloor +1.$
\\\textbf{(P2.2.)} Let the sequence $\beta: \:\: b_{s}=0\:\:\text{mod}\:\: m \:\: \text{for}\:\: s \geq 1.$  The number of the block of the sequence $\:\: \dfrac{b_{1}}{m}, \dfrac{b_{2}}{m}, \dfrac{b_{3}}{m},...$ is $\:\: L(mn)$.
\\\textbf{(P2.3.)} Let a sequences $\widetilde{\beta}$ be the union of $m$ rows of the sequence $\beta$:
\\$\widetilde{b_{1}}=b_{1}+b_{2}+... \: b_{m},\:\: \widetilde{b_{2}}=b_{m+1}+b_{m+2}+... \: b_{2m}, 
\:\: \widetilde{b_{3}}=b_{2m+1}+b_{2m+2}+... \: b_{3m},\:\: ...$
\\Then $ \widetilde{L}(n)=\lfloor\dfrac {L(n)+m-1}{m}\rfloor. $ 
\\\\Let's examine some special cases of the sequence  $\beta$.
\\\\{\bf Example 2.0.} Let $p_{0} \in \mathbb{Z}^{+}, \: b_{s}=p_{0}$. Using (1), (2) and (3) we get
$$\\B(s) = p_{0}s, \quad x(n)=\dfrac {n}{p_{0}}, \quad L(n)=\Bigl\lceil \dfrac {n}{p_{0}} \Bigl\rceil,$$
$$\\R(n)= n- p_{0}(L(n)-1).$$
\\{\bf Example 2.1.} Let the partitioning sequence $\beta$ is linear function 
$b_{s}=p_{1}s + p_{0}, 
\\\text{where} \:\: p_{0} \in \mathbb{Z}, \:\: p_{1}  \in \mathbb{Z}^{+}$.
Using (1), (2) and (3) we get
$$\\B(s) = p_{1}\dfrac {s(s+1)}{2}+p_{0}s.$$ 
\begin{equation}  
	L(n)=\Bigl\lceil \dfrac{-2p_{0}-p_{1}+\sqrt{8 n p_{1}+(2p_{0}+p_{1})^2}}{2p_{1}}\Bigl\rceil.	
\end{equation}
$$\\R(n)= n - p_{1}\dfrac {(L(n)-1)L(n)}{2}+p_{0}(L(n)-1),$$  
\\\\Let $p_{1}=5 \:\: \text{and} \:\: p_{0}=2, \:\: \text{then} \:\:  L(n)=\lceil\sqrt{n+9}-3\rceil:$ 
$\\1,1,1,1,1,1,1,$
$\\2,2,2,2,2,2,2,2,2,$
$\\3,3,3,3,3,3,3,3,3,3,3,$ 
$\\ \quad .\quad .\quad.\\$
\\{\bf Example 2.1.1.} This is a special case of the previous example $p_{0}=0, \:\:  b_{s}=p_{1}s$.
$$\\B(s) = p_{1}\dfrac {s(s+1)}{2},$$  
$$\\L(n)=\Bigl\lceil \dfrac{-p_{1}+\sqrt{8 n p_{1}+p_{1}^2}}{2p_{1}}\Bigl\rceil.$$	
\\For $p_{1}=1$ we  obtain  the regular array and popular formula $\seqnum{A002024}$
$$\\L(n)=  \Bigl\lceil \dfrac{-1+\sqrt{8 n +1}}{2}\Bigl\rceil.$$	
\\For $p_{1}=2$ we get irregular array and the formula $\seqnum{A000194}$
$$\\L(n)=  \Bigl\lceil \dfrac{-1+\sqrt{4 n +1}}{2}\Bigl\rceil.$$	
\\We can also solve this problem by using (P2.1). 
$$\\L(n)=  \Bigl\lceil \dfrac{-1+\sqrt{8u +1}}{2}\Bigl\rceil, \:\: \text{where} \:\: u=\Bigl\lfloor \frac{n-1}{p_{1}} \Bigr\rfloor +1.$$
\\Article [4] presents an alternative method 
$$\\L(n)=\Bigl\lceil\sqrt{\Bigl\lceil{\frac{2n}{p_{1}}}\Bigl\rceil}+\dfrac{1}{2}\Bigl\rceil - 1.$$
%
%
\\{\bf Example 2.1.2.} Cantor's diagonalization is a well-known for numbering infinite arrays. In this example, we propose a generalization of the Cantor numbering method for two adjacent diagonals. A pair of neighboring diagonals are combined into one block. 
The sequences $\alpha = \xi$. The partitioning sequence $\beta$ is $b_{s}= 4s-1,
p_{1}=4, \:\: p_{0}=-1,\:\:  \seqnum{A004767}: \:\:  3, 7, 11, 15, 19,... $ Then
$$\\L(n) = \Bigl\lceil \frac{-1 + \sqrt{8n + 1}}{4}\Bigl\rceil.$$
The partial sums $B(s)=s(2s+1)$ is the sequence of second hexagonal numbers [3] , \seqnum{A014105}.
%
%
\\\\{\bf Example 2.1.3.} Let $ d \in \mathbb{Z}^{+},\:\: d>1 $. We shall combine $d$ diagonals into one block, starting with the first diagonal. The sequence $\alpha = \xi$. Then 
$$\\ b_{s}= d^2s -\dfrac {d(d-1)}{2},  \quad B(s)=\dfrac{ds(ds+1)}{2}.$$ 
\\Using (4) for $p_{1}=d^2$ and $p_{0}=-\dfrac {d(d-1)}{2}$ we get
$$\\L(n)=\Bigl \lceil\frac {-1+\sqrt{8n-7}}{2d}\Bigl \rceil, $$ 
\\A second way to solve this problem is to use (P2.3.). 
$$\\L(n)=\lfloor\dfrac {t+d}{d}\rfloor, \;\text{where}\; t=\lfloor \dfrac{\sqrt{8n-7}-1}{2} \rfloor .\\$$ 
\\For $d=3 \; b_{s}=9s-3$ is the sequences $\seqnum{A017233}$ and we obtain
$\\1,1,1,1,1,1,$
$\\2,2,2,2,2,2,2,2,2,2,2,2,2,2,2,$
$\\3,3,3,3,3,3,3,3,3,3,3,3,3,3,3,3,3,3,3,3,3,3,3,3,$
$\\ \quad .\quad .\quad.\\$
%
%
\\{\bf Example 2.1.4.} Now we shall change example 2.1.3. by forming blocks of $d$ adjacent diagonals, starting from the second diagonal, $\alpha = \xi$. Then 
$$\\b_{1}=1, \; b_{s}= d^2(s-1) -\dfrac {d(d-3)}{2} \; \text {for} \; s>1, $$
The partitioning sequence $\beta$ is not linear function. 
$$\\B(0)=0, \; B(s)=\dfrac{(d(s-1)+1)(d(s-1)+2)}{2} \; \text {for} \; s>1. $$ 
Using [3] we obtain
$$\\L(n)=\Bigl \lceil\frac {2d-3+\sqrt{8n+1}}{2d}\Bigl \rceil. $$ 
\\We can solve this problem using a modified version of (P2.3.). Let a sequences $\widetilde{\beta}:
\:\: \widetilde{b_{1}}=b_{1}, \:\: \widetilde{b_{2}}=b_{2}+... \: b_{m+1},\:\: \widetilde{b_{3}}=b_{m+2}+b_{m+3}+... \: b_{2m+1}, \: ...$. Then
$$\\L(n)=\lfloor\dfrac {t+d-1}{d}\rfloor+1, \:\: \text{where}\ \:\: t=\lfloor \dfrac{\sqrt{8n-7}-1}{2} \rfloor .\\$$
\\For $d=3$ the sequence $\beta$  is $b_{1}=1, \; b_{s}=9(s-1), \; \text{for} \; s>1$ and we get
$\\1,$
$\\2,2,2,2,2,2,2,2,2,$
$\\3,3,3,3,3,3,3,3,3,3,3,3,3,3,3,3,3,3,$
$\\ \quad .\quad .\quad.\\$
%
%
%
\\{\bf Example 2.2.} Let the partitioning sequence $\beta$ is quadratic function 
\\$b_{s}=p_{2}s^2 + p_{1}s + p_{0}, \:\: \text{where} \:\: p_{0}, \:\: p_{1} \in \mathbb{Z},\:\: p_{2}  \in \mathbb{Z}^{+}$. 
$$\\B(s)=p_{2}\frac{s(s+1)(2s+1)}{6}+p_{1}\frac{s(s+1)}{2}+p_{0}s.$$
For the cubic equation
$$\\2p_{2}x^3 + (3p_{2}+3p_{1})x^2 + (p_{2}+3p_{1}+6p_{0})x - 6n = 0$$
\\we use Cardano's formula [5].
\begin{equation}  
L(n) =
\\ \Biggl\lceil-\frac{p_{1}+p_{2}}{2p_{2}} - \frac{U}{3\cdot2^{2/3}\cdot p_{2}
\sqrt[3]{V+{\sqrt{4U^3+V^2}}}}
+\frac{1}{6 \cdot 2^{1/3}\cdot p_{2}}\sqrt[3]{V+{\sqrt{4U^3+V^2}}}\Biggl\rceil,
\end{equation}
\\where
$$\\U= 3(- 3p_{1}^2 + 12p_{0}p_{2} - p_{2}^2),$$
$$\\V = 54(- p_{1}^3 + 6p_{0}p_{1} + 12np_{2}^2 + 6p_{0}p_{2}^2 + p_{1}p_{2}).$$
$$\\R(n)=n- p_{2} \dfrac {(L(n)-1)L(n)(2(L(n)-1)+1}{6}-p_{1}\dfrac {(L(n)-1)L(n)}{2} - p_{0}(L(n)-1).$$
\\{\bf Example 2.2.1.} This is a special case of the previous example $p_{2}=1, 
\\p_{1}=0, \:\: p_{0}\ge 0, \:\:  b_{s}=s^2+p_{0}$. Using (1) and (5) we get
$$\\B(s) =\frac{s(s+1)(2s+1)}{6}+ p_{0}s,$$  
$$\\U=36p_{0}-3, \quad V=648n+324p_{0}. \quad \text{then}$$
\\ The discriminant $\Delta=-(4(36p_{0}-3)^3+(648n+324p_{0})^2)<0$
and so the cubic equation has one real root and two non-real complex conjugate roots.
$$\\L(n) =
\\ \Biggl\lceil-\frac{1}{2} - \frac{36p_{0}-3}{3\cdot2^{2/3} 
	\sqrt[3]{648n+324p_{0}+{\sqrt{4(36p_{0}-3)^3+(648n+324p_{0})^2}}}}$$
$$\\+\frac{1}{6 \cdot 2^{1/3}}\sqrt[3]{648n+324p_{0}+{\sqrt{4(36p_{0}-3)^3+(648n+324p_{0})^2}}}\Biggl\rceil,$$
\\For $p_{0}=0$ we obtain the formula for $\seqnum{A074279}:$ 
$$\\ L(n)= \Biggl\lceil {\frac{1}{2} \left( -1 + \frac{1}{3^{1/3}W} + \frac{W}{3^{2/3}} \right) }\Biggl\rceil,$$ 
$$\\ \text{where} \:\: W=(108 n+\sqrt{3}\sqrt{-1+3888 n^2})^{1/3}.$$
\\For $p_{0}=1$ we obtain $L(n):$
$\\1,1,$
$\\2,2,2, 2,2,$
$\\3,3,3, 3,3,3, 3,3,3, 3,$
$\\ \quad .\quad .\quad.\\$
\\{\bf Example 2.2.2.} Let the sequence $\beta$ is quadratic function with the coefficients 
$$\\p_{2}=\dfrac{m-2}{2}, \quad p_{1}=-\dfrac{m-4}{2}, \quad p_{0}=0, \quad m \in \mathbb{Z}^{+}, \: m \ge 3.$$
\\Then  $b_{s}=\dfrac{(m-2)s^2-(m-4)s}{2}$ 
form the sequence of polygonal numbers [3]. Using (1) and (5) we get
$$\\B(s)=(m-2)\frac{s(s+1)(2s+1)}{12} - (m-4)\frac{s(s+1)}{4}.$$
The cubic equation takes the form $(2m-4)x^3 + 6x^2-(2m-10)x -12n=0.$ Then
$$\\U=-156+84m-12m^2, $$
$$\\V=-2592+1512 m-216 m^2+5184 n-5184 m n+1296 m^2 n,$$
$$\\L(n) =
\Biggl\lceil-\frac{1}{m-2} - \frac{U}{3\cdot2^{2/3}\cdot (m-2)
	\sqrt[3]{V+{\sqrt{4U^3+V^2}}}}$$
$$\\+\frac{1}{6 \cdot 2^{1/3}\cdot (m-2)}\sqrt[3]{V+{\sqrt{4U^3+V^2}}}\Biggl\rceil,$$
\\For $m > 19,$ the cubic polynomial is in casus irreducibilis, with three distinct real roots. Therefore, we must use a trigonometric solution to find the roots.
\\For $m=5$ the sequence  $b_{s}$ is the sequence of pentagonal numbers $\seqnum{A000326}$.  
\\Then $L(n):$
$\\1,$
$\\2,2,2,2,2,$
$\\3,3,3,3,3,3,3,3,3,3,3,3,$
$\\ \quad .\quad .\quad.\\$
\\{\bf Example 2.2.3.} Let the sequence $\beta$ is quadratic function with the coefficients 
$$\\p_{2}=\dfrac{m}{2}, \quad p_{1}=-\dfrac{m}{2}, \quad p_{0}=1, \quad m \in \mathbb{Z}^{+}.$$
\\Then  $b_{s}=m\dfrac{s^2-s}{2}+1$ 
form the sequence of centered polygonal numbers [3]. Using (1) and (5) we get
$$\\B(s)=m\frac{s(s+1)(2s+1)}{12} - m\frac{s(s+1)}{4} + s.$$
The cubic equation takes the form $mx^3 + (6-m)x-6n=0.$ Then
$$\\L(n) =  \Biggl\lceil-\frac{2^{1/3} (6-m)}{\sqrt[3]{162 m^2 n+\sqrt{108 (6-m)^3 m^3+26244 m^4 n^2}}}+$$
$$\\ \frac{{\sqrt[3]{162 m^2 n+\sqrt{108 (6-m)^3 m^3+26244 m^4 n^2}}}}{3\cdot 2^{1/3}\cdot m}\Biggl\rceil$$
\\For  $m > 24,$ the cubic polynomial is in casus irreducibilis, with three distinct real roots. Consequently, we employ trigonometric solution to find the roots.
\\For $m=5$ the sequence $b_{s}$ is the sequence of centered pentagonal numbers $\seqnum{A005891}.$ Then $L(n):$
$\\1,$
$\\2,2,2,2,2,2,$
$\\3,3,3,3,3,3,3,3,3,3,3,3,3,3,3,3,$
$\\ \quad .\quad .\quad.\\$
%
%
%
\\{\bf Example 2.3.} Let the partitioning sequence $\beta$ is cubic function 
\\$b_{s}=p_{3}s^3 + p_{2}s^2 + p_{1}s + p_{0}, \:\: \text{where} \:\: p_{0}, p_{1}, p_{2} \in \mathbb{Z}, \:\: p_{3}  \in \mathbb{Z}^{+}$. 
$$\\B(s)=p_{3}\frac{s^2(s+1)^2}{4}+p_{2}\frac{s(s+1)(2s+1)}{6}+p_{1}\frac{s(s+1)}{2}+p_{0}s.$$
There are formulas [5],[6] for solving the 4th degree equation
$$\\3p_{3}x^4 + (6p_{3}+4p_{2})x^3 + (3p_{3}+6p_{2}+6p_{1})x^2 + (12p_{0}+6p_{1}+2p_{2})x - 12n = 0.$$
A different approach is to use numerical solutions of equation.
%
%
%
\\\\{\bf Example 2.3.1.} This is a special case of the previous example $p_{3}=1, 
\\p_{1}=0, \:\: p_{0}\ge 1, \:\:  b_{s}=s^3+p_{0}$. For $p_{0}=1$ we get the equation 
\\ $x^{2}(x+1)^{2}+4x-4n=0.$ Using (3) we obtain $L(n):$
$\\1, 1,$
$\\2,2,2, 2,2,2, 2,2,2,$
$\\3,3,3, 3,3,3, 3,3,3, 3,3,3, 3,3,3, 3,3,3, 3,3,3, 3,3,3, 3,3,3, 3,$
$\\ \quad .\quad .\quad.\\$
%
%
%
\\{\bf Example 2.3.2.} Let the sequence $\beta$ is cubic function with the coefficients 
$$\\p_{3}=\dfrac{m-2}{6}, \quad p_{2}=\dfrac{1}{2}, \quad p_{1}=-\dfrac{m-5}{6}, \quad p_{0}=0, \quad m \in \mathbb{Z}^{+}, \: m \ge 3.$$ Then  
$$\\b_{s}=\dfrac{1}{6}s(s+1)((m-2)s-(m-5))$$ 
\\form the sequence of pyramidal numbers [3]. For $m=5$ we get the sequence of pentagonal pyramidal numbers $\seqnum{A002411}$ and the equation 
$$\\x(x^{3} + 4x^{2}+ 4x+1) -12n =0.$$
\\Using (3) we obtain $L(n):$ 
$\\1,$
$\\2,2,2,2,2,2,$
$\\3,3,3,3,3,3, 3,3,3,3,3,3, 3,3,3,3,3,3$
$\\ \quad .\quad .\quad.\\$
%
%
%
\\{\bf Example 2.4.} Let the partitioning sequence $\beta$ is 
$$b_{s}=(m-1)m^{s-1} \: \text{for} \: m > 1 \:\text{and}\: s\geq 1.$$ 
Using (1), (2) and (3) we get
$$\\B(s) = m^{s} - 1, \quad L(n)=\lceil \log_{m}(n+1) \rceil,$$
$$\\R(n)= n - m^{\lceil \log_{m}(n+1) \rceil-1}+1.$$
\\The sequences $\seqnum{A029837}$ and $\seqnum{A081604}$ are examples of sequences generated by $m=2$ and $m=3$, respectively.
\section{Intra-Block Permutation of Integer Positive Numbers}
Let $\beta$ is the partitioning sequence. The sequence $\xi$ is written as irregular array read by rows:
$\\ B(0)+1, B(0)+2, \: ... \: , B(0)+b_{1},$
$\\ B(1)+1,B(1)+2, \: ... \:, B(1)+b_{3},$
$\\ \quad .\quad .\quad.\\$
$\\B(k)+1, B(k)+2, \: ... \:, B(k)+b_{k+1}$.
$\\ \quad .\quad .\quad.\\$
\\{\bf Definition 3.1.} A sequence $\alpha \in \mathcal{A}^{+}$ is called an intra-block permutation of integer positive numbers if it maps each block $(B(k)+1, B(k)+2, \: ... \:, B(k)+b_{k+1})$ to itself.
\\This means that each block of the sequences $\alpha$
$\\a_{B(k)+1}, a_{B(k)+2}, ..., a_{B(k)+b_{k+1}}$ is a permutation of the numbers 
$\\B(k)+1, B(k)+2, \: ... \:, B(k)+b_{k+1}$.
\\\\Denote by $\pi(n)$ a permutation of the first $n$ natural numbers $(p_{1},p_{2},p_{3},...p_{n})$. 
The group of all permutations  $\pi(n)$ is denoted by $S_n$, is called the symmetric group of degree $n$. 
The order of a permutation $\pi$, denoted by $o(\pi)$, is defined as the
smallest positive integer $m$ such that $\pi^{m} = id$ [5].
\\The set of numbers 
$$\\a_{B(k)+1}-B(k), a_{B(k)+2}-B(k), ..., a_{B(k)+b_{k+1}}-B(k)$$
is a permutation $\pi(b_{k+1}).$
\\\\The sequence $\alpha$ is determined by the sequence $\beta$ and the set of permutations \\$\pi(b_{1}),\pi(b_{2}),\pi(b_{3})... \:\: $ Let $\alpha \circ \alpha$ is self-composition $\alpha$($\alpha$) of the sequence $\alpha$ [7]. This operation is equivalent to multiplying permutations.
$$\\ \pi(b_{1}) \circ \pi(b_{1}), \:\: \pi(b_{2}) \circ \pi(b_{2}),\:\: \pi(b_{3}) \circ \pi(b_{3}),...$$ 
The sequence $\xi$ consists of identity permutations.
\\\\{\bf Definition 3.2.} The order of a sequence $\alpha$, denoted by $o(\alpha)$, is the smallest positive integer $m$ such that $m$ times self-composition $\alpha^{m} = \xi$.
\\\\The following properties hold.
\\\textbf{(P3.1.)} The sequence $\alpha$ is permutation of the natural numbers.
\\\textbf{(P3.2.)} The order of $\alpha \:\:\: o(\alpha)=LCM (o(\pi(b_{1}),o(\pi(b_{2}),o(\pi(b_{3}),...  )$.
\\\textbf{(P3.3.)} The sequences $\alpha, \alpha^{2}, \alpha^{3}, ...$ form a cyclic group.  
\\\\Let a sequence $\gamma$: $g_{1}, g_{2}, g_{3},...\quad \in$ $\mathcal{A}^{+}$ such that
$\\g_{1} + g_{2} + ... + + g_{m_{1}} =b_{1},$
$\\g_{m_{1}+1} + g_{m_{1}+2} + ... + g_{m_{1}+m_{2}} =b_{2},$
$\\g_{m_{1}+m_{2}+1} + g_{m_{1}+m_{2}+2} + ... + + g_{m_{1}+m_{2}+m_{3}} =b_{3},$
$\\ \quad .\quad .\quad.\\$
Thus, the sequence $\gamma$ partitions the sequence $\xi$ into blocks, such that each block of the sequence $\beta$ is a collection of disjoint blocks of $\gamma$, whose union is the block $\beta$.
We denote by $\gamma \leq \beta.$ Let a sequence $\mu \in \mathcal{A}^{+}$ is an intra-block permutation of integer positive numbers for partitioning sequence $\gamma.$  
\\\\\textbf{(P3.4.)} The sequences $\mu$ is intra-block permutation for the partitioning sequence $\beta.$  
\\\textbf{(P3.5.)} The set of sequences $\mathcal{A}^{+}$,  equipped with a binary relation $\leq$, form partially ordered set [8]. The minimal element is the sequence $(1,1,1,...) \quad \seqnum{A000012}.$
\\\textbf{(P3.6.)} The sequences $\alpha \circ \mu$ is intra-block permutation for the partitioning sequence $\beta.$  
\\\\In all examples in this section, we shall use the partitioning sequence from example 2.1.2. $\:\: \beta: b_{s}= 4s-1 \:\: \text{for} \:\: s \geq 1$. All permutations  $\pi(b_{1}),\pi(b_{2}),\pi(b_{3})... \:\:$ have odd length. The sequence $\alpha = \xi$. 
%
%
\\\\{\bf Example 3.1.} Terms of the $\pi(n): \:\: p_{i}=R^{'}(i).$ The order of permutations $o(\pi(b_{s}))=2 \:\: \text{for} \:\: s \geq 1.$ So $o(\alpha)=2$ and the 
sequence $\alpha$ is self-inverse permutation of the natural numbers. The sequence as irregular array begins
$\\3,2,1,$
$\\10,9,8,7,6,5,4,$
$\\21,20,19,18,17,16,15,14,13,12,11,$
$\\ \quad .\quad .\quad.\\$
%
%
\\{\bf Example 3.2.} The formula for terms of the $\pi$:
\\$p(i)=\begin{cases}
R^{'}(i),&\text{if $R^{'}(i)\geq R(i)+1$},\\
R(i)-\Bigl \lfloor \dfrac{R(i)+R^{'}(i)-1}{2} \Bigr \rfloor, &\text{if $R^{'}(i) < R(i)+1$}.\\
\end{cases}$
\\The order of permutations $o(\pi(b_{1}))=3, \:\: o(\pi(b_{s}))=12 \:\: \text{for} \:\: s \geq 2.$
\\Thus $o(\alpha)=12$. The sequence begins
$\\3,1,2,$
$\\10,9,8,4,5,6,7,$
$\\21,20,19,18,17,11,12,13,14,15,16,$
$\\ \quad .\quad .\quad.\\$
%
%
\\{\bf Example 3.3.} The formula for terms of the $\pi$:
\\$p_{i} =\begin{cases}
	\Bigl\lfloor\ \dfrac{4L(i)-1}{2} \Bigl\rfloor+R(i)+1,&\text{if $R(i) > R^{'}(i)$},\\
\\	R(i) - \Bigl\lfloor \dfrac{4L(i)-1}{2} \Bigl\rfloor,&\text{if $R(i) \leq R^{'}(i)$.}
\end{cases}$
\\The order of permutations $o(\pi(b_{s}))=b_{s} \:\: \text{for} \:\: s \geq 1.$ So the sequence $\alpha$ has infinite order. Another formula: 
$$\\a(n)= \dfrac {(i+j-1)^2+i-j+3 +2(i + j - 1)(-1)^{i + j}}{2},$$ 
\\where
$$\\i=n-\dfrac {t(t+1)}{2}, \quad j= \dfrac {t^2+3t+4}{2}-n, \quad t=\lfloor \dfrac{\sqrt{8n-7}-1}{2} \rfloor .$$ 	 
\\The start of the sequence $\alpha$:
$\\3,1,2,$
$\\8,9,10,4,5,6,7,$
$\\17,18,19,20,21,11,12,13,14,15,16$
$\\ \quad .\quad .\quad.\\$
\section{Generalized reluctant sequences}
{\bf Definition 4.1.} Let sequences $\alpha$ $\in$ $\mathcal{A}$ and $\omega$ $\in$ $\mathcal{A}^{+}$. The sequence $\omega$ is called the reluctant sequence of sequence $\alpha$, if $\omega$ is the triangle array read by rows, with row number $k$ coinciding with the first $k$ elements of the sequence $\alpha$ [2].
\\Formula for a reluctant sequence is:
$$\\\omega(n)=a_{m},\:\:\: \text{where} \:\:\: m=n-\dfrac {t(t+1)}{2}, \:\:\:  t=\lfloor \dfrac{\sqrt{8n-7}-1}{2} \rfloor. $$
\\{\bf Definition 4.2.} Let sequences $\alpha$ $\in$ $\mathcal{A}$ and $\omega$ $\in$ $\mathcal{A}^{+}$. The sequence $\omega$ is called the reverse reluctant sequence of sequence $\alpha$, if $\omega$ is the triangle array read by rows, with row number $k$ coinciding with the first $k$ elements of the sequence $\alpha$ in reverse order [2].
\\Formula for a reverse reluctant sequence is:
$$\\ \omega(n)=a_{m}, \:\:\: \text{where} \:\:\: m=\dfrac {t^2+3t+4}{2}-n, \:\:\: t=\lfloor \dfrac{\sqrt{8n-7}-1}{2} \rfloor. $$
\\\\Let $q \in\mathbb{Z}^{+}$. Denote by $(a_{k+1},a_{k+2},a_{k+3},... \; a_{k+m})^q$
\\$q$ times concatenation of the block $a_{k+1},a_{k+2},a_{k+2},...\; a_{k+m}:$
\\$\underbrace{a_{k+1},a_{k+2},a_{k+3},...\; a_{k+m},\:\:\: a_{k+1},a_{k+2},a_{k+3},...\; a_{k+m}, \:\:\:  ... \:\:\: a_{k+1},a_{k+2},a_{k+3},... \; a_{k+m}}_{\text{$q$ times}} $ 
\\\\{\bf Definition 4.3.} Let a sequence $\alpha$: $a_{1}, a_{2}, a_{3},...\:\in \mathcal{A}.$ A sequences $\beta$: 
\\$b_{1}, b_{2}, b_{3},...$ $\in$ $\mathcal{A}^{+}$ \:\text{and} \:$q \in\mathbb{Z}^{+} $. The sequence $\omega$ is called the generalized reluctant sequence of sequences $\alpha$ if $\omega$ is irregular array read by rows:
\begin{equation}
   \begin{split}
&\qquad(a_{1}, a_{2}, \dots \: a_{b_{1}})^{q},\\ 
&\qquad(a_{1}, a_{2}, \dots \: a_{b_{1}}, a_{b_{1}+1}, \dots \: a_{b_{1}+b_{2}})^{q},\\
&\qquad(a_{1}, a_{2}, \dots \: a_{b_{1}}, a_{b_{1}+1}, \dots \: a_{b_{1}+b_{2}}, a_{b_{1}+b_{2}+1},a_{b_{1}+b_{2}+2},\dots \: a_{b_1+b_2+b_3})^{q},\\   
&\qquad \dots 
   \end{split}
\end{equation}
\\{\bf Definition 4.4.} Let a sequence $\alpha$: $a_{1}, a_{2}, a_{3},...\:\in \mathcal{A}.$ A sequences $\beta$: 
\\$b_{1}, b_{2}, b_{3},...$ $\in$ $\mathcal{A}^{+}$ \:\text{and} \:$q \in\mathbb{Z}^{+} $. The sequence $\omega^{'}$ is called the generalized reverse reluctant sequence of sequences $\alpha$ if $\omega^{'}$ is irregular array read by rows:
\begin{flalign*}
&\qquad (a_{b_1}, a_{b_1-1}, \dots, a_{1})^{q}, &&\\
&\qquad (a_{b_1+b_2}, a_{b_1+b_2-1}, \dots \, a_{b_1}, a_{b_1-1}, \dots \, a_{1})^{q}, &&\\
&\qquad (a_{b_1+b_2+b_3}, a_{b_1+b_2+b_3-1}, \dots \, a_{b_1+b_2}, a_{b_1+b_2-1}, \dots \, a_{b_1}, a_{b_1-1}, \dots \, a_{1})^{q}, &&\\
&\qquad \dots &&\\
\end{flalign*}
As an illustration, consider the following examples. Let $\alpha = \xi$, a partitioning sequences $\gamma$ is increasing $g_{s} < g_{s+1}, \:\:\text{for} \:\:\ s \geq 1 $. Then
\\The sequence $R(n)$: 
$\\ 1,2, ...\: g_{1},$ 
$\\ 1,2, ...\: g_{1}, ...\: g_{2},$ 
$\\ 1,2, ...\: \: g_{1}, ...\: g_{2}, ...\: g_{3},$ 
$\\ \quad .\quad .\quad.\\$
is generalized reluctant sequence of sequences $\xi$ \: for \: $q=1$. 
\\Similarly, for the sequence $R^{'}(n)$:
$\\ g_{1}, g_{1}-1, ...\: 1,$ 
$\\ g_{2}, g_{2}-1, ...\: g_{1}, g_{1}-1, ...\: 1,$ 
$\\ g_{3}, g_{3}-1, ...\: g_{2}, g_{2}-1, ...\: g_{1}, g_{1}-1, ...\: 1,$ 
$\\ \quad .\quad .\quad.\\$
is generalized reverse reluctant sequence of sequences $\xi$ for same $\gamma \:$ and $\: q=1$. 
\\\\If the sequence $\beta = \seqnum{A000012}$ and $q=1$ generalized reluctant sequence becomes reluctant sequence $\seqnum{A002260}.$  Similarly, for the same sequence $\beta$ and $q$ generalized reverse reluctant sequence becomes reverse reluctant sequence  $\seqnum{A004736}.$  
\\\\There are some examples generalized reluctant sequence 
\\for $q=1 \quad  \text{and} \quad  \beta: 
\\b_{1}=1, \:\: b_{s}=2 \:\: \text{for} \:\: s \geq 2 \:\: \seqnum{A071797},
\\b_{1}=1, \:\: b_{s}=2s-1 \:\: \text{for} \:\: s \geq 2 \:\: \seqnum{A064866},
\\b_{1}=1, \:\: b_{s}=2^{s-2} \:\: \text{for} \:\: s \geq 2 \:\: \seqnum{A062050},$
\\for $ \:\: q=2 \:\: \quad  \text{and} \quad  \beta: 
\\b_{s}=1, \:\: s \geq 1 \:\: \seqnum{A122197}.$
\\The example of generalized reverse reluctant sequence is $\seqnum{A080883}
\\ \text{for}  \:\: q=1 \:\:  \text{and} \:\: \beta: \:\: b_{1}=1, \:\: b_{s}=2 \:\: \text{for} \:\: s \geq 2.$  
\\\\Let's create a formula to calculate $L$. Denote by $\zeta: c_{1}, c_{2}, c_{3},...$ the partitioning sequence for the array (6), where $c_{s}=qB(s).$  Denote by  $C(s)$ partial sums $\zeta$:
$$\\C(s)=0, \:\: C(s)= c_{1}+c_{2}+...+c_{s}.$$ 
\\Using (2) and (3) we get 
$$L(n)=\lceil y(n) \rceil,$$ 
\\where $y(n)$ is the largest  root of the equation $C(x)=n,$ 
$$\\R(n) = n- C(L(n)-1), \:\: R^{'}(n) = C(L(n))+1-n.$$
Let's develop a formula for finding the term of the array (6). The term $\omega(n)$ is located in the row $L(n)$ at the place $R(n)$. The row $L(n)$ contains the block of terms 
$$\\a_{1}, a_{2}, ... \: a_{B(L(n))}$$
repeated $q$ times. This row is numbered $R(n)$ from 1 to $c_{L(n)}=qB(L(n))$. 
\\Then generalized reluctant sequence of sequences $\omega(n) =a(m), \:\: \text{where}$ 
$$\\m=1+ (R(n)-1) \:\: \text{mod} \:\:B(L(n)).$$
\\Similarly, the generalized reverse reluctant sequence of sequences $\omega^{'}(n) =a(m^{'}), \:\: \text{where}$ 
$$\\m^{'}=1+(R^{'}(n)-1) \:\: \text{mod} \:\:B(L(n)).$$
%
%
\\\\{\bf Example 4.0.} Let $p, q \in \mathbb{Z}^{+}, \text{the sequences} \:\: \beta$: $\:\: b_{s}=p$ $\:\: \text{for} \:\: s\geq 1$.
\\Then 
$$\\B(s)=ps, \:\: c_{s}=pqs, \:\: C(s)=pq\dfrac{s(s+1)}{2},$$ 
$$\\L(n)=\Bigl\lceil \dfrac{-pq+\sqrt{8npq+p^2q^2}}{2pq} \Bigr\rceil,$$
$$\\R(n)=n-pq\dfrac{(L(n)-1)L(n)}{2},\quad R^{'}(n)=pq\dfrac{L(n)(L(n)+1)}{2} +1 - n .$$
\\We get for generalized reluctant sequence $\omega$ and reverse reluctant sequence $\omega^{'}$:
$$\\m=1+ (R(n)-1) \:\: \text{mod} \:\:  pL(n),$$ 
$$\\m^{'}=1+ (R^{'}(n)-1) \:\: \text{mod} \:\:  pL(n).$$ 
\\Let the sequence $\alpha = \xi$,  $\beta$: $b_{s}=2$ $\: \text{for} \: s\geq 1$  $\:\text{and} \:\: q=3$. 
\\Then generalized  reluctant sequence  $\omega$:
$\\1,2,\:\: 1,2,\:\: 1,2,  
\\1,2,3,4, \:\: 1,2,3,4, \:\: 1,2,3,4,
\\1,2,3,4,5,6, \:\: 1,2,3,4,5,6,\:\: 1,2,3,4,5,6
\\ \quad .\quad .\quad.\\$
Generalized  reverse reluctant sequence  $\omega^{'}$:
$\\2,1\:\: 2,1\:\: 2,1,  
\\4,3,2,1, \:\: 4,3,2,1  \:\: 4,3,2,1,
\\6,5,4,3,2,1 \:\: 6,5,4,3,2,1,\:\: 6,5,4,3,2,1,
\\ \quad .\quad .\quad.\\$
%
%
\\\\{\bf Example 4.1.} Let $p_{1}, q \in \mathbb{Z}^{+}, \text{the sequences} \:\: \beta$: $\:\: b_{s}=p_{1}s\:\: \text{for} \:\: s\geq 1$  $\:\: \text{and} \:\: q \geq 1$.
\\Then 
$$\\B(s)=p_{1}\dfrac{s(s+1)}{2}, \:\: c_{s}=p_{1}q\dfrac{s(s+1)}{2}, \:\: C(s)=p_{1}q\dfrac{s(s+1)(s+2)}{6},$$ 
Using Cardano’s formula [5] we get
$$L(n)=\Bigl \lceil-1 + \dfrac{p_{1}q}{\sqrt[3]{3}U}+\dfrac{U}{\sqrt[3]{3^2}p_{1}q} \Bigr \rceil,$$
$$\text{where}\:\: U= \left(27np_{1}^2q^2 + \sqrt{3} \sqrt{243n^2p_{1}^4q^4 - p_{1}^6q^6}\right)^{1/3}.$$
$$\\R(n)=n-p_{1}q\dfrac{(L(n)-1)L(n)(L(n)+1)}{6},$$   $$\\R^{'}(n)=p_{1}q\dfrac{L(n)(L(n)+1)(L(n)+2)}{6}+1-n $$
\\We get for generalized reluctant sequence $\omega$ and reverse reluctant sequence $\omega^{'}$:
$$\\m=1+ (R(n)-1) \:\: \text{mod} \:\:  p_{1}\dfrac{L(n)(L(n)+1)}{2},$$ 
$$\\m^{'}=1+ (R^{'}(n)-1) \:\: \text{mod} \:\:  p_{1}\dfrac{L(n)(L(n)+1)}{2}.$$ 
\\If the sequence $\alpha = \xi$,  $\beta$: $b_{s}=2s$ $\: \text{for} \: s\geq 1$  and $ \:\: q=3$. 
\\Then generalized  reluctant sequence  $\omega$:
$\\1,2,\:\: 1,2,\:\:  1,2,
\\1,2,3,4,5,6, \:\: 1,2,3,4,5,6,\:\: 1,2,3,4,5,6
\\1,2,3,4,5,6,7,8,9,10,11,12, \: 1,2,3,4,5,6,7,8,9,10,11,12, \: 1,2,3,4,5,6,7,8,9,10,11,12
\\ \quad .\quad .\quad.\\$
Generalized  reverse reluctant sequence  $\omega^{'}$:
$\\2,1, \:\: 2,1,\:\: 2,1,
\\6,5,4,3,2,1, \:\: 6,5,4,3,2,1, \:\: 6,5,4,3,2,1, 
\\ 12,11,10,9,8,7,6,5,4,3,2,1, \:\: 12,11,10,9,8,7,6,5,4,3,2,1, \:\: 12,11,10,9,8,7,6,5,4,3,2,1
\\ \quad .\quad .\quad.\\$
%
%
\\\\{\bf Example 4.2.} Let $p, q \in \mathbb{Z}^{+}, \:\: p \geq  2, \:\: q \geq 1, \text{the sequences} \:\: \beta$: $\:\: b_{1}=p, \:\: b_{s}=p^{s}-p^{s-1}$ $\:\: \text{for} \:\: s\geq 2$. Then
$$\\B(s)=p^{s}, \:\: c_{s}=qp^{s}, \:\: C(s)=\dfrac{pq(p^{s}-1)}{p-1},$$ 
$$\\L(n)=\Bigl\lceil \log_{p}\Bigl (\dfrac{n(p-1)}{pq}+1 \Bigr ) \Bigr\rceil,$$ 
$$\\R(n)=n-pq\dfrac{p^{L(n)-1}-1}{p-1},\quad R^{'}(n)=pq\dfrac{p^{L(n)}-1}{p-1}+1 - n .$$
\\We get for generalized reluctant sequence $\omega$ and reverse reluctant sequence $\omega^{'}$:
$$\\m=1+ (R(n)-1) \:\: \text{mod} \:\:  p^{L(n)},$$ 
$$\\m^{'}=1+ (R^{'}(n)-1) \:\: \text{mod} \:\:  p^{L(n)} .$$ 
\\Let the sequence $\alpha = \xi$,  $\beta:\:\: b_{1}=2, \:\: b_{s}=2^{s}-2^{s-1} \: \text{for} \: s\geq 2$  $\:\text{and} \:\: q=3$. 
\\Then generalized  reluctant sequence  $\omega$:
$\\1,2, \:\: 1,2, \:\: 1,2,
\\1,2,3,4, \:\: 1,2,3,4,\:\: 1,2,3,4,
\\1,2,3,4,5,6,7,8,\:\:1,2,3,4,5,6,7,8,\:\:1,2,3,4,5,6,7,8,
\\ \quad .\quad .\quad.\\$
Generalized  reverse reluctant sequence  $\omega^{'}$:
$\\2,1,\:\: 2,1,\:\:2,1,
\\4,3,2,1, \:\: 4,3,2,1, \:\: 4,3,2,1,
\\8,7,6,5,4,3,2,1, \:\: 8,7,6,5,4,3,2,1, \:\: 8,7,6,5,4,3,2,1,
\\ \quad .\quad .\quad.\\$

E-mail:putievskiy@gmail.com
\end{document}